\documentclass[12pt]{amsart}
\usepackage{amssymb,latexsym}
\usepackage{amsfonts}
\usepackage{amsmath}
\usepackage[margin=3cm]{geometry}
\usepackage[colorlinks,linkcolor=blue,anchorcolor=blue,citecolor=blue]{hyperref}
\usepackage{algorithm}
\usepackage{enumerate}
\usepackage{algpseudocode}
\usepackage{verbatim}
\usepackage{graphicx}

\newcommand\Q{{\mathbb Q}}
\newcommand\Z{{\mathbb Z}}
\newcommand\cP{{\mathcal P}}
\newcommand\cQ{{\mathcal Q}}
\newcommand\ord{\mathrm{ord}}

\newtheorem{theorem}{Theorem}[section]
\newtheorem{lemma}[theorem]{Lemma}

\newtheorem{conjecture}[theorem]{Conjecture}

\newtheorem{corollary}[theorem]{Corollary}
\newtheorem{proposition}[theorem]{Proposition}

\theoremstyle{remark}

\numberwithin{equation}{section}

\begin{document}

\title[Irregular primes]{Irregular primes with respect to Genocchi numbers and 
Artin's primitive root conjecture}

\author{Su Hu}
\address{Department of Mathematics, South China University of Technology, Guangzhou 510640, China}
\email{mahusu@scut.edu.cn}

\author{Min-Soo Kim}
\address{Division of Mathematics, Science, and Computers, Kyungnam University, 7(Woryeong-dong) kyungnamdaehak-ro, Masanhappo-gu, Changwon-si,
Gyeongsangnam-do 51767, Republic of Korea}
\email{mskim@kyungnam.ac.kr}

\author{Pieter Moree}
\address{Max-Planck-Institut f\"ur Mathematik, Vivatsgasse 7, D-53111 Bonn, Germany}
\email{moree@mpim-bonn.mpg.de}

\author{Min Sha}
\address{Department of Computing, Macquarie University, Sydney, NSW 2109, Australia}
\email{shamin2010@gmail.com}

\dedicatory{Dedicated to the memory of Prof. Christopher Hooley (1928-2018)}



\begin{abstract}
We introduce and study a variant of Kummer's notion of (ir)regularity of primes which we call G-(ir)regularity and is based on
Genocchi rather
than Bernoulli numbers. 
We say that an odd prime $p$ is G-irregular 
if it divides at least one of the 
Genocchi numbers $G_2,G_4,\ldots, G_{p-3}$, 
and G-regular otherwise.  
We show that, as in Kummer's case, G-irregularity is related to the 
divisibility of some class number. 
Furthermore, we obtain some results on the distribution of G-irregular primes. In particular, we show that each
primitive residue class contains infinitely many G-irregular 
primes and establish non-trivial lower bounds for
their number up to a given bound $x$ as $x$ tends
to infinity.
As a byproduct, we obtain some results on the
distribution of primes in arithmetic progressions with a prescribed near-primitive root.
\end{abstract}

\maketitle

\section{Introduction}
\label{Intro}

\subsection{The classical case}

The $n$-th Bernoulli polynomial $B_n(x)$ 
 is  implicitly defined as the coefficient of 
$t^n$ in the generating function 
\begin{equation*}
\frac{te^{xt}}{e^t-1}=\sum_{n=0}^{\infty}B_{n}(x) \frac{t^n}{n!}, 
\end{equation*}
where $e$ is the base of the natural logarithm. 
For $n\ge 0$ the Bernoulli numbers $B_n$ are defined by 
$B_{n}=B_{n}(0).$ 
It is well-known that $B_0=1$, and $B_n = 0$ for every odd integer $n > 1$. 

In this paper, $p$ always denotes a prime. 
An odd prime $p$ is said to be \textit{B-irregular}  
if $p$ divides the numerator of at least one of the Bernoulli numbers $B_{2},B_{4},\ldots,B_{p-3}$, 
and \textit{B-regular} otherwise. 
The first twenty B-irregular primes are
\begin{align*}
& 37, 59, 67, 101, 103, 131, 149, 157, 233, 257, 263,  \\ 
& 271,  283, 293, 307, 311, 347, 353, 379, 389.
\end{align*}

The notion of B-irregularity has an important application in algebraic number theory. Let $\mathbb{Q}(\zeta_{p})$ be the $p$-th cyclotomic field, 
and $h_p$ its class number. 
Denote by $h_p^{+}$ the class number of $\Q(\zeta_p+\zeta_p^{-1})$ and
put
$h_p^{-}=h_p / h_p^{+}.$
Kummer proved that $h_p^{-}$ is an integer 
 (now called the \textit{relative class number} of $\Q(\zeta_p)$) 
and gave the following characterization (see \cite[Theorem 5.16]{Wa}). 
(Here and in the sequel, for any integer $n \ge 1$, $\zeta_n$ denotes an $n$-th primitive root of unity.)  

\begin{theorem}[Kummer]
\label{thm:Kummer}
An odd prime $p$ is B-irregular if and only if $p\mid h_{p}^{-}$.
\end{theorem}
Kummer showed that if $p\nmid h_p^{-},$ then
the Diophantine equation $x^p+y^p=z^p$ does not have an integer solution
$x,y,z$ with $p$ coprime to $xyz,$ cf. \cite[Chapter 1]{Wa}.

We remark that for any odd prime $p$, $p\mid h_{p}^{-}$ if and only if $p\mid h_{p}$; see \cite[Theorem 5.34]{Wa}. The class number $h_{p}^{+}$ is rather more
elusive than $h_{p}^{-}$. Here Vandiver 
made the conjecture that $p\nmid h_{p}^{+}$ for any odd prime $p$.  

Jensen \cite{Jensen} was the 
first to prove that there are infinitely many B-irregular primes.
More precisely, he showed that there are infinitely many B-irregular primes
not of the form $4n + 1$. 
This was generalized by Montgomery \cite{Mon}, 
who showed that  4 can be replaced by any integer greater than 2. 
To the best of our knowledge, 
the following result due to Mets\"ankyl\"a \cite{Met2} has not been 
improved upon. 

\begin{theorem}[Mets\"ankyl\"a \cite{Met2}]
\label{thm:Met}
Given an integer $m>2$, let $\Z_m^*$ be the multiplicative group of the residue classes modulo $m$, and let $H$ be a proper subgroup of $\Z_m^*$. 
Then, there exist infinitely many B-irregular primes not lying in the residue classes in $H$. 
\end{theorem}

Let ${\mathcal P}_B$ be the set of B-irregular primes.
Carlitz \cite{Carlitz} gave a simple proof of 
the infinitude of this set, 
and recently Luca, Pizarro-Madariaga and Pomerance \cite[Theorem 1]{Luca} made this more quantitative by showing that 
\begin{equation}
\label{eq:Luca}
{\mathcal P}_B(x)  \ge (1+o(1)) \frac{\log\log x}{\log\log\log x}
\end{equation}
as $x \to \infty$. (Here and in the sequel, if
$S$ is a set of natural numbers, then $S(x)$ denotes
the number of elements in $S$ not exceeding $x.$)
Heuristical arguments suggest a much stronger 
result (consistent with numerical data; see, for instance, \cite{BH,HHO}).

\begin{conjecture}[Siegel \cite{Siegel}]   \label{conj:Siegel1}
Asymptotically we have
 $$
 \cP_B(x)\sim \left( 1-\frac{1}{\sqrt{e}}\right) \pi(x),
 $$
 where $\pi(x)$ denotes the prime counting function. 
 \end{conjecture}
 
The reasoning behind this conjecture is as follows.
We assume that the numerator of $B_{2k}$ is not divisible
by $p$ with probability $1-1/p.$ Therefore, assuming the 
independence of divisibility by distinct primes, we expect
that $p$ is B-regular with probability
$$
\left( 1-\frac{1}{p}\right)^{\frac{p-3}{2}},
$$
which with increasing $p$ tends to $e^{-1/2}$. 

For any positive integers $a,d$ with $\gcd(a,d)=1$, let $\cP_B(d,a)$ 
be the set of B-irregular primes congruent to $a$ modulo $d$. 
We pose the following conjecture, 
which suggests that B-irregular primes are uniformly distributed in arithmetic progressions. 
It is consistent with numerical data; see Table \ref{tab:B-irre}. 

\begin{conjecture}  \label{conj:Siegel2}
For any positive integers $a,d$ with $\gcd(a,d)=1$, asymptotically we have
 $$
 \cP_B(d,a)(x)\sim \frac{1}{\varphi(d)} \left(1-\frac{1}{\sqrt{e}}\right) \pi(x), 
 $$
 where $\varphi$ denotes Euler's totient function. 
\end{conjecture}

Although we know that there are infinitely many 
 B-irregular primes, it is still not known whether
 there are infinitely many B-regular primes!

\subsection{Some generalizations}

In \cite{Carlitz} Carlitz studied (ir)regularity with respect to Euler numbers. 

The Euler numbers $E_n$ are a sequence  of integers defined by the relation 
$$
\frac{2}{e^t+e^{-t}} = \sum_{n=0}^{\infty} E_n \frac{t^n}{n!}. 
$$
It is easy to see
that $E_0=1$, and $E_n=0$ for any odd $n\ge 1$. 
Euler numbers, just as Bernoulli numbers, can be defined via special polynomial values.  
The Euler polynomials $E_{n}(x)$ are implicitly defined as the coefficient of 
$t^n$ in the generating function
\begin{equation*}  \label{Eu-pol}
\frac{2e^{xt}}{e^t+1}=\sum_{n=0}^\infty E_n(x)\frac{t^n}{n!}.
\end{equation*}
Then, $E_n = 2^n E_n(1/2), n=0,1,2,\ldots.$  

An odd prime $p$ is said to be \textit{E-irregular} 
if it divides at least one of the integers  $E_{2},E_{4},\ldots,E_{p-3}$, 
and \textit{E-regular} otherwise. 
The first twenty E-irregular primes are 
\begin{align*}
& 19, 31, 43, 47, 61, 67, 71, 79, 101, 137, 139, \\
& 149, 193, 223, 241, 251, 263, 277, 307, 311.
\end{align*}
Vandiver \cite{Van} proved that if $p$ is E-regular, then 
the Diophantine equation $x^p+y^p=z^p$ does not have an integer solution
$x,y,z$ with $p$ coprime to $xyz.$ 
This criterion makes it possible to discard 
many of the exponents Kummer could not handle:
$$
37, 59, 103, 131, 157, 233, 257, 271, 283, 293,\ldots. 
$$

Carlitz \cite{Carlitz} showed that there are infinitely many E-irregular primes. 
Luca, Pizarro-Madariaga and Pomerance in \cite[Theorem 2]{Luca} 
showed that the number of  E-irregular primes up to
$x$ satisfies the same lower bound as in \eqref{eq:Luca}. 
Regarding their distribution, we currently only know that there are infinitely many E-irregular primes 
not lying in the residue classes $\pm 1~({\rm mod~}8)$, which was proven by Ernvall \cite{Ern}.  

Like the B-(ir)regularity of primes, their E-(ir)regularity also relates to the
divisibility of class numbers of cyclotomic fields; see \cite{Ern2,Ern3}. 

Let $\cP_E$ be the set of E-irregular primes, and let $\cP_E(d,a)$ 
be the set of E-irregular primes congruent to $a$ modulo $d$.  
It was conjectured and tested in \cite[Section 2]{EM} that:  

\begin{conjecture}  \label{conj:E-irre1}
Asymptotically we have
 $$
 \cP_E(x)\sim \left(1-\frac{1}{\sqrt{e}}\right)\pi(x). 
 $$
 \end{conjecture}

As in the case of B-irregular primes, we pose the following conjecture. 

\begin{conjecture}  \label{conj:E-irre2}
 For any positive integers $a,d$ with $\gcd(a,d)=1$, asymptotically we have
 $$
 \cP_E(d,a)(x)\sim \frac{1}{\varphi(d)} \left(1-\frac{1}{\sqrt{e}}\right)\pi(x). 
 $$
 \end{conjecture}
\noindent This conjecture is consistent with computer 
calculations; 
see Table \ref{tab:E-irre}.  

Later on, Ernvall \cite{Ern2,Ern3} introduced $\chi$-irregular primes 
and proved the infinitude of such irregular primes for any Dirichlet character $\chi$, 
including B-irregular primes and E-irregular primes as special cases.  
In addition, Hao and Parry \cite{HP} defined $m$-regular primes for any square-free integer $m$, 
and Holden \cite{Holden} defined irregular primes by using the values of zeta functions of 
totally real number fields. 

In this paper, we 
introduce a new kind of irregular prime based on Genocchi 
numbers and study 
their distribution in 
detail.

\subsection{Regularity with respect to Genocchi numbers}
\label{sec:Euler}

The Genocchi numbers $G_n$ are defined by the relation  
$$
\frac{2t}{e^t+1} = \sum_{n=1}^{\infty} G_n \frac{t^n}{n!}. 
$$
It is well-known that $G_1=1$, $G_{2n+1}=0$ for $n \ge 1$, and $(-1)^nG_{2n}$ is an odd positive integer.  
The Genocchi numbers $G_n$ are related to Bernoulli numbers $B_n$ by the formula 
\begin{equation}  \label{eq:BnGn}
G_n = 2(1-2^n)B_n. 
\end{equation}
In view of the definitions of $G_n$ and $E_n(x)$, we directly obtain
\begin{equation}  \label{eq:EnGn}
G_n = nE_{n-1}(0), \quad n \ge 1.  
\end{equation}
 
In analogy with Kummer and Carlitz, we here define an odd prime $p$ to be \textit{G-irregular} 
if it divides at least one of the integers $G_2,G_4,\ldots, G_{p-3}$, and
\textit{G-regular} if it does not.
The first twenty G-irregular primes are  
\begin{align*}
& 17, 31, 37, 41, 43, 59, 67, 73, 89, 97, 101, 103, \\
& 109, 113, 127, 131, 137, 149, 151, 157. 
\end{align*}
Clearly, if an odd prime $p$ is B-irregular, then it is also G-irregular. 

Recall that a \textit{Wieferich prime} is an odd prime $p$ such that 
$2^{p-1}\equiv 1~({\rm mod~}p^2)$, which arose in the study of Fermat's Last Theorem. 
So, if an odd prime $p$ is a Wieferich prime, then $p$ divides $G_{p-1}$, 
and otherwise it does not divide $G_{p-1}$.  
Currently there are only two Wieferich prime known, namely 
$1093$ and $3511$.
If there are further ones, they are larger than $6.7 \times 10^{15}$; see \cite{DK}.  
Both $1093$ and $3511$ are G-irregular primes. 
However, $1093$ is B-regular, and $3511$ is B-irregular. 

As in the classical case, the G-irregularity of primes can 
be linked to the divisibility of some class numbers of cyclotomic fields.  
Let $S$ be the set of infinite places of $\Q(\zeta_p)$ and $T$ the set of places above the prime 2. 
Denote by $h_{p,2}$ the \textit{$(S,T)$-refined class number} of $\Q(\zeta_p)$. 
Similarly, let $h_{p,2}^{+}$ be the refined class number of $\Q(\zeta_p+\zeta_p^{-1})$ with respect to its infinite places and places above the prime 2
(for the definition of the refined class number of global fields,
we refer to Gross \cite[Section 1]{Gross} or Hu and Kim \cite[Section 2]{HK}). 
Define 
$$
h_{p,2}^{-} = h_{p,2} / h_{p,2}^{+}.
$$
It turns out that $h_{p,2}^{-}$ is an integer
(see \cite[Proof of Proposition 3.4]{HK}).

\begin{theorem}
\label{thm:class}
 Let $p$ be an odd prime. 
 Then, if $p$ is G-irregular, we have $p\mid h_{p,2}^{-}$. 
 If furthermore  $p$ is not a Wieferich prime, the converse is also true. 
\end{theorem}

\subsubsection{Global distribution of G-irregular primes}

Let $g$ be a non-zero integer. For an odd prime $p\nmid g$,  let ${\rm ord}_p(g)$ be the multiplicative
order of $g$ modulo $p$, that is, the smallest positive integer 
$k$ such that $g^{k}\equiv 1~({\rm mod~}p)$.

\begin{theorem}  
\label{thm:GB}
An odd prime $p$ is G-regular if and only 
if it is B-regular and satisfies $\ord_p(4)=(p-1)/2$.  
\end{theorem}

Note that $\ord_p(4)\mid (p-1)/2.$ Using quadratic reciprocity it is not difficult to
show (see Proposition~\ref{prop:splitsing}) that if $p\equiv 1~({\rm mod~}8)$,  
then $\ord_p(4)\ne (p-1)/2$.
 Thus Theorem~\ref{thm:GB} has the following corollary.

\begin{corollary}
\label{cor:allepriem}
Primes $p$ satisfying $p\equiv 1~({\rm mod~}8)$ 
are G-irregular.
\end{corollary}

Although the B-irregular  primes are very mysterious, the set
of primes 
$p$ such that $\ord_p(4)=(p-1)/2$ is far less so.
Its distributional properties are analyzed in detail in 
Proposition \ref{prop:main}. The special case $a=d=1$ in
combination with Theorem \ref{thm:GB}, yields
the following estimate.
\begin{theorem}
\label{totalEirregular}
Let $\cP_G$ be the set of G-irregular primes.  
Let $\epsilon>0$ be arbitrary and fixed. Then we have, for every $x$ sufficiently large, 
$$
 \cP_G(x)  > \left( 1-\frac{3}{2}A-\epsilon\right) \frac{x}{\log x},$$
with
$A$ the Artin constant 
\begin{equation}
\label{Artinconstantdef}
A= \prod_{\textrm{prime $p$}}\left( 1-\frac{1}{p(p-1)} \right)  = 0.3739558136192022880547280543464\ldots .
\end{equation} 
\end{theorem}
\noindent Note that $1-3A/2=0.4390662795\ldots.$

Using Siegel's heuristic, 
one arrives at the following conjecture.

\begin{conjecture} 
\label{con:global}
Asymptotically we have
$$
\cP_G(x)\sim \left( 1-\frac{3A}{2\sqrt{e}}\right) \pi(x)~(
 \approx 0.6597765\cdot\pi(x)).
$$
\end{conjecture}
 
 Some numerical evidence for Conjecture \ref{con:global} is given in Section \ref{sec:data}.  
The heuristic behind this conjecture is straightforward.
Under the Generalized Riemann Hypothesis (GRH) it can be shown that the set 
of primes 
$p$ such that
$\ord_p(4)=(p-1)/2$ has density $3A/2$ (Proposition~\ref{prop:main} with $a=d=1$). By
Siegel's heuristic one expects a fraction ${3A}/{(2\sqrt{e})}$ of these to be B-regular.
The conjecture follows on invoking Theorem~\ref{thm:GB}.

\subsubsection{G-irregular primes in prescribed arithmetic progressions}

Using Proposition \ref{prop:main} we can give a 
non-trivial lower bound for the number of
G-irregular primes in a prescribed arithmetic
progression. Recall that if $S\subseteq T$ 
are sets of natural numbers, then the relative density of
$S$ in $T$ is defined as 
$${\lim}_{x\rightarrow \infty}\frac{S(x)}{T(x)},$$
if this limit exists.
We use the notations gcd and lcm for
greatest common divisor, respectively least common multiple, but often will
write $(a,b)$,  rather than $\gcd(a,b)$. 
We also use the ``big O" notation $O$, and  we write $O_\rho$ to emphasize the dependence of the implied
constant on some parameter (or a list of parameters) $\rho$.   

\begin{proposition}
\label{prop:main}
Given two coprime positive integers $a$ and $d,$ we put 
\begin{equation}
\label{cqda}
\cQ(d,a)=\{p>2:~p\equiv a~({\rm mod~}d),~\ord_p(4)=(p-1)/2\}.
\end{equation}
Let $\epsilon$ be arbitrary and fixed. 
Then,  for every $x$ sufficiently large we have
\begin{equation}
\label{eq:underGRH2}
\cQ(d,a)(x)<
\frac{(\delta(d,a)+\epsilon)}{\varphi(d)}\frac{x}{\log x} 
\end{equation}
with
$$
\delta(d,a)=c(d,a)R(d,a)A,
$$
where
\begin{equation*}
R(d,a)=2\prod_{p\mid (a-1,d)}\left( 1-\frac{1}{p}\right) 
\prod_{p\mid d}\left( 1+\frac{1}{p^2-p-1}\right) ,
\end{equation*}
and
$$
c(d,a)=
\begin{cases}
3/4  & \text{if~}4\nmid d;\\
1/2  & \text{if~}4\mid d,8\nmid d,~a\equiv 1~({\rm mod~}4);\\
1  & \text{if~}4\mid d,8\nmid d,~a\equiv 3~({\rm mod~}4);\\
1  & \text{if~}8\mid d,~a\not\equiv 1~({\rm mod~}8);\\
0  & \text{if~}8\mid d,~a\equiv 1~({\rm mod~}8),
\end{cases}
$$
is the relative density of the primes
$p\not\equiv 1~({\rm mod~}8)$ in the
set of primes $p\equiv a~({\rm mod~}d).$

Under  GRH, we have
\begin{equation}
\label{eq:underGRH}
\cQ(d,a)(x)= 
\frac{\delta(d,a)}{\varphi(d)}\frac{x}{\log x}
+
O_d\left(\frac{x\log\log x}{\log^2x}\right).
\end{equation}
\end{proposition}

We discuss the connection of this result and 
Artin's primitive root conjecture in Section \ref{Artinmethods}.  
A numerical demonstration of the
estimate \eqref{eq:underGRH} is given in Section~\ref{sec:data} for 
some choices of $a$ and $d$. 
 By Proposition~\ref{prop:splitsing}, in case $8 | d$ and $a\equiv 1~({\rm mod~}8)$,  
we in fact have $\cQ(d,a)=\emptyset$ and so $\cQ(d,a)(x)=0$. 

Combination of Theorem~\ref{thm:GB} and 
Proposition~\ref{prop:main} yields directly the following result.

\begin{theorem}
\label{thm:mainArtin}
Given two coprime positive integers $a$ and $d$,  we put
$$
\cP_G(d,a)=\{p:\, p\equiv a~({\rm mod~}d) \textrm{ and  $p$ is G-irregular}\}.
$$
Let $\epsilon$ be arbitrary and fixed. 
For every $x$ sufficiently large, we have
\begin{equation}
\label{dainequality}
\cP_G(d,a)(x)>
\frac{(1-\delta(d,a)-\epsilon)}{\varphi(d)}\frac{x}{\log x},
\end{equation}
where $\delta(d,a)$ is defined in Proposition~\ref{prop:main}.

Under GRH we have
\begin{equation}
\cP_G(d,a)(x)\ge 
\frac{(1-\delta(d,a))}{\varphi(d)}\frac{x}{\log x}
+
O_d\left( \frac{x \log\log x}{\log^{2} x} \right) .
\end{equation}  
\end{theorem}
 
 Note that 
the inequality \eqref{dainequality} (with $a=d=1$) yields Theorem \ref{totalEirregular} as a
special case.

An easy analysis 
(see \eqref{eq:delta} in Section \ref{size}) shows that $1-\delta(d,a)>0,$ and so 
we obtain the following corollary, which can be compared with Theorem~\ref{thm:Met}. 

\begin{corollary}
\label{cor:cor}
 Each primitive residue class contains a subset of G-irregular primes having positive density.
\end{corollary}

Moreover, by Corollary~\ref{cor:allepriem} in case  $a\equiv 1~({\rm mod~}8)$ and $8|d$ 
we have $1-\delta(d,a)=1$ and in fact 
$\cP_G(d,a)(x)=\pi(x;d,a)$,    
 where 
 $$
 \pi(x;d,a) = \#\{p\le x:p\equiv a~({\rm mod~}d)\}.
 $$   
In the remaining cases we have $\delta(d,a) > 0$, 
and so $1-\delta(d,a) < 1$. 
However, $1-\delta(d,a)$ can be arbitrarily close to 1 (see Proposition~\ref{inf-sup}), 
and the same holds 
for the relative density of $\cP_G(d,a)$ by Theorem~\ref{thm:mainArtin}.  

Corollary  \ref{cor:cor} taken by itself is not a deep result and can 
be easily proved directly, see Moree and Sha \cite[Proposition 1.6]{MS}.

The reasoning that leads us
to Conjecture \ref{con:global} in
addition with the assumption that
B-irregular primes are equidistributed over residue classes with a fixed modulus, 
suggests that the following conjecture might be true.

\begin{conjecture} 
\label{con:local}
Given two coprime positive integers $a$ and $d$, asymptotically we have
 $$\cP_G(d,a)(x)\sim \left( 1-\frac{\delta(d,a)}{\sqrt{e}}\right) \pi(x;d,a),$$ where $\delta(d,a)$ is defined
 in Proposition \ref{prop:main}.
 \end{conjecture}
 
  Numerical evidence for Conjecture \ref{con:local} is presented in Section \ref{sec:data}.  
Note that this conjecture implies Conjecture \ref{con:global} (choosing $a=d=1$). 
On observing that $\delta(a,d)=0$ if and only if 
$8\mid d$ and $a\equiv 1~({\rm mod~}8)$, 
it also implies the following conjecture. 

\begin{conjecture} 
\label{con:local2}
Consider the subset of G-regular primes in the primitive residue
class $a~({\rm mod~}d)$.
It has a positive density, provided 
we are not in the case $8\mid d$ and
$a\equiv 1~({\rm mod~}8)$. 
\end{conjecture}

\section{Preliminaries}

In this section, we gather some results which are used later on. 

\subsection{Elementary results}

For a primitive Dirichlet character $\chi$ with an odd conductor $f,$
the generalized Euler numbers $E_{n,\chi}$ are defined by (see \cite[Section 5.1]{KH1})
\begin{equation*}
2\sum_{a=1}^{f}\frac{(-1)^a\chi(a)e^{at}}{e^{ft}+1}=\sum_{n=0}^{\infty}E_{n,\chi}\frac{t^n}{n!}. 
\end{equation*}

For any odd prime $p$, let $\omega_p$ be the Teichm\"uller character of $\Z/p\Z.$ Any multiplicative character of $\Z/p\Z$ is of the form $\omega_p^k$ for some $1\le k \le p-1$. 
In particular, the odd characters are $\omega_p^{k}$ with $k=1,3,\ldots,p-2$.

The following lemma, formulated and proved in the
notation of \cite{KH1}, is an analogue of a well-known result for the generalized Bernoulli numbers; see \cite[Corollary 5.15]{Wa}. 
\begin{lemma}
\label{p4} 
Suppose that $p$ is an odd prime and $k,n$ 
are non-negative integers. Then $E_{k,\omega_p^{n-k}}\equiv E_{n}(0) ~({\rm mod~}p).$
\end{lemma}
\begin{proof} 
By \cite[Proposition 5.4]{KH1}, for arbitrary integers $k,n\geq 0$, we have
\begin{equation}\label{(1)} 
E_{k,\omega_p^{n-k}}=\int_{\mathbb{Z}_{p}}\omega_p^{n-k}(a)a^{k}d\mu_{-1}(a).
\end{equation}
By \cite[Proposition 2.1 (1)]{KH1}, $E_{n}(0)$ also can be expressed as a $p$-adic integral,
namely
\begin{equation}\label{(2)} 
E_{n}(0)=\int_{\mathbb{Z}_{p}}a^{n}d\mu_{-1}(a).
\end{equation}
Since $\omega_p(a)\equiv a ~({\rm mod~}p)$,  we have
\begin{equation}\label{(3)}
\begin{aligned}
\omega_p^{n-k}(a)\equiv a^{n-k} ~({\rm mod~}p) \qquad \textrm{and} \qquad \omega_p^{n-k}(a)a^{k}\equiv a^{n} ~({\rm mod~}p).
\end{aligned}
\end{equation}
From \eqref{(1)}, \eqref{(2)} and \eqref{(3)}, we deduce that
$$
E_{k,\omega_p^{n-k}}-E_{n}(0)=\int_{\mathbb{Z}_{p}}(\omega_p^{n-k}(a)a^{k}-a^{n})d\mu_{-1}(a)\equiv
0~({\rm mod~}p).\qedhere
$$
\end{proof}

Recall that $\cQ(d,a)$ is defined 
in \eqref{cqda}. For ease of notation
we put
\begin{equation}
\label{Qdef}
\cQ=\cQ(1,1)=
\{p>2:\ord_p(4)=(p-1)/2\}.
\end{equation}
For the understanding of the distribution of the 
primes in $\cQ,$ it turns out to be very useful to
consider their residues modulo $8$.

\begin{proposition}
\label{prop:splitsing}
For $j=1,3,5,7$ we put
$\cQ_j=\cQ(8,j).$
We have $$\cQ=\cQ_1\cup \cQ_3 \cup \cQ_5 \cup \cQ_7,$$
with $\cQ_1=\emptyset,$
and, for $j=3,5,$
$$\cQ_j=\{p: p \equiv j~({\rm mod~}8),~\ord_p(2)=p-1\},$$
and, furthermore,
$$\cQ_7=\{p: p \equiv 7~({\rm mod~}8),~ \ord_p(2)=(p-1)/2\}.$$
\end{proposition}

\begin{proof}
If $p\equiv 1~({\rm mod~}8)$, then
by quadratic reciprocity $2^{(p-1)/2}
\equiv 1~({\rm mod~}p)$,  and we
conclude that $\ord_p(4)\mid (p-1)/4$ 
and hence $\cQ_1=\emptyset$. 

Note that
$$
\ord_p(4)=
\begin{cases}
\ord_p(2) & {\rm ~if~}\ord_p(2){\rm ~is~odd};\\
\ord_p(2)/2 & {\rm ~otherwise}.
\end{cases}
$$
In case $p\equiv \pm 3 ~({\rm mod~}8),$ 
we have $2^{(p-1)/2}
\equiv -1~({\rm mod~}p)$, and so
$\ord_p(2)$ must be even. The 
assumption that $p$ is in $\cQ$ now
implies that $\ord_p(2)=2\cdot \ord_p(4)=
p-1.$
In case $p\equiv 7 ~({\rm mod~}8),$ 
we have $2^{(p-1)/2}
\equiv 1~({\rm mod~}p)$, and so
$\ord_p(2)$ must be odd. The 
assumption that $p$ is in $\cQ$ now
implies that $\ord_p(2)=
\ord_p(4)=(p-1)/2$.  
\end{proof}

\subsection{The size of $\delta(d,a)$}
\label{size}
In this section we study the 
extremal behaviour of the quantity $\delta(d,a)$ defined in Proposition~\ref{prop:main}. 
We put
$$
G(d)=\prod_{p \mid d}\left(1+\frac{1}{p^2-p-1}\right),\qquad    F(d) = \frac{\varphi(d)}{d}G(d).
$$
An easy calculation gives that
$$
G(d)=\frac{1}{A}\prod_{p \nmid d}\left(1-\frac{1}{p(p-1)}\right)=\frac{1}{A}\left(1+O(\frac{1}{q})\right),
$$
where $q$ is the smallest prime not dividing $d$.  
Trivially, $G(d)<1/A$, and $G(d)<1/(2A)$ when $d$ is odd.

It is a classical result (see, for instance, \cite[Theorem 13.14]{Apostol}), that
$$
\liminf_{d \to \infty} \frac{\varphi(d)}{d}  \log\log d = e^{-\gamma}, 
$$ 
where $\gamma$ is the Euler-Mascheroni constant ($\gamma=0.577215664901532\ldots$).
The proof is in essence an application of
Mertens' theorem (see \cite[Theorem 13.13]{Apostol})
\begin{equation}
    \label{mertens}
    \prod_{p\le x}\left( 1-\frac{1}{p}\right) \sim \frac{e^{-\gamma}}{\log x}. 
\end{equation}
An easy variation of the latter proof yields 
$$
\liminf_{d \to \infty} AF(d)  \log\log d = e^{-\gamma}. 
$$

Recall that
\begin{equation*}
R(d,a)=2G(d)\prod_{p \mid b}
\left(1-\frac{1}{p}\right)=2G(d)\frac{\varphi(b)}{b}, \text{ with } b=(a-1,d).
\end{equation*}
Note that $$2F(d)\le R(d,a)\le 2G(d)/(2,d)<1/A,$$ 
and hence $\delta(d,a)=0$ or
\begin{equation}  
    \label{eq:delta}
0<AF(d)\le \delta(d,a)\le 2AG(d)/(2,d)<1.
\end{equation}

\begin{proposition}  \label{inf-sup}
We have 
$$
\liminf_{d\to \infty} \min_{\substack{1\le a < d \\ (a,d)=1\\ \delta(d,a)>0}} \delta(d,a)\log\log d  = e^{-\gamma}
\qquad \text{and} \qquad 
\limsup_{d\to \infty} \max_{\substack{1\le a < d \\ (a,d)=1}} \delta(d,a)  = 1. 
$$
\end{proposition}
\begin{proof}
From the above remarks it follows that 
the limit inferior and superior
are $\ge e^{-\gamma},$ respectively $\le 1.$
We consider two infinite families of pairs $(a,d)$
to show that these bounds are actually sharp.

Let $n\ge 3$ be 
arbitrary. Put $d_n=\prod_{3\le p \le n} p$.  
We have $c(4d_n,1)=1/2$ and
$$
\delta(4d_n,1)=(1+o(1))\prod_{2\le p\le n}(1-1/p), \quad (n\rightarrow \infty)
$$ 
by Proposition \ref{prop:main}. 
Using Mertens' theorem \eqref{mertens}
and the prime number theorem, we deduce that 
$$
\delta(4d_n,1)\sim \frac{e^{-\gamma}}{\log n}\sim 
\frac{e^{-\gamma}}{\log \log (4d_n)},\quad (n\rightarrow \infty),
$$
and so the limit inferior actually equals 
$e^{-\gamma}.$

Put
$$a_n=
\begin{cases}
2+3d_n& \text{~if~}d_n\equiv 7 ~({\rm mod}~8);\\
2+d_n & \text{~otherwise}.
\end{cases}
$$
We have $a_n\not\equiv 1 ~({\rm mod}~8),$ $1\le a_n<8d_n,$ $(a_n,8d_n)=1,$ and
$(a_n-1,8d_n)$ is a power of two. We infer
that $R(8d_n,a_n)=1/A+O(1/n)$ and $c(8d_n,a_n)=1,$
and so $\delta(8d_n,a_n)=1+O(1/n),$
showing that the limit superior equals $1$.  
\end{proof}

The two constructions in the above proof
are put to the test 
in Table~\ref{tab:cRA}.
The table also gives an idea of how fast 
the lower bound $1-\delta(4d_n,1)$ 
for the relative density of the set $\cP_G(4d_n,1)$ 
established in Theorem~\ref{thm:mainArtin}, tends to $1$.  

\begin{table}   
\centering
\caption{Some values of $\delta(4d_n,1)$ and $\delta(8d_n,a_n)$}
\label{tab:cRA}
\begin{tabular}{|c|c|c|c|c|c|c|}  
\hline
$n$ & $10^3$ & $10^4$ & $10^5$ & $10^6$ & $10^7$    \\ \hline 
$\delta(4d_n,1)\approx $ & 0.080954  &  0.060884 & 0.048752 & 0.040638 & 0.034833   \\ \hline

$\delta(4d_n,1)e^{\gamma}\log \log (4d_n)\approx $ & 
0.989659  & 0.997633  & 0.999422 &  0.999851 &  0.999960  \\ \hline

$\delta(8d_n,a_n) \approx $ & 0.999872  &  0.999990 & 0.999999 & 0.9999999 & 0.99999999  \\ \hline

\end{tabular}
\end{table}

\section{Some results related to Artin's primitive root conjecture}
\label{Artinmethods}
It is natural to wonder whether the set
$\cQ,$ see \eqref{Qdef},  is an infinite set or not. This is closely
related to Artin's primitive root conjecture 
stating that if $g\ne -1$ or a square, then 
infinitely often $\ord_p(g)=p-1$ (which is maximal by Fermat's little theorem).
In case $g$ is a square, the maximal order is
$(p-1)/2$ and one can wonder whether this happens
infinitely often. 
If this is so for $g=4,$ then our
set $\cQ$ is infinite. We now go into a bit
more technical detail.

We say that a set of primes $\cP$ has
density $\delta(\cP)$ and satisfies a Hooley type
estimate, if
\begin{equation}
\label{hooleytype}
 \cP(x)=\delta(\cP)\frac{x}{\log x}+
O\left(\frac{x\log\log x}{\log^2x}\right),
\end{equation} 
where the implied constant may depend on $\cP$. 

Let $g\not \in\{-1,0,1\}$ be 
an integer.  
Put
$${\mathcal P}_g=\{p: \ord_p(g)=p-1\}.$$
Artin in 1927 conjectured that this set, when $g$ is not a square, is infinite and also conjectured a density for it.
To this day, this conjecture is open; see \cite{Moree} for a survey.
Hooley \cite{Hooley1} proved in 1967
that if the Riemann Hypothesis holds for the number fields
$\Q(\zeta_n,g^{1/n})$ with all square-free $n$ (this is a weaker
form of the GRH), then the
estimate \eqref{hooleytype} holds for the set $\cP_g$  with 
\begin{equation*}
\delta(g):=\delta(\cP_g)=\sum_{n=1}^\infty
\frac{\mu(n)}{[\Q(\zeta_n,g^{1/n}):\Q]},
\end{equation*} 
where $\mu$ is the M{\"o}bius function; 
and he also showed that $\delta(g)/A$ is rational 
and explicitly determined its value,
with $A$ the Artin constant (see \eqref{Artinconstantdef}). 
For example, in case $g=2$ we have $\delta(2)=A$.   

By the Chebotarev density theorem,  the density of primes 
$p\equiv 1 ~({\rm mod~}n)$ such that $\ord_p(g)\mid (p-1)/n$ 
is equal to $1/[\Q(\zeta_n,g^{1/n}):\Q]$.  
Note that in order to ensure that
$\ord_p(g)=p-1$,  it is enough
to show that there is no prime $q$ 
such that $\ord_p(g)\mid (p-1)/q$. 
By inclusion and exclusion we
are then led to expect that
the set $\cP_g$ has natural 
density $\delta(g)$. The problem
with establishing this rigorously is
that the Chebotarev density theorem
only allows one to take finitely many
splitting conditions into account. Let
us now consider which result we can obtain on 
restricting to the primes $q\le y.$ 
Put
\begin{equation}
\label{classicalhooley2} 
\delta_y(g)=\sum_{P(n)\le y}
\frac{\mu(n)}{[\Q(\zeta_n,g^{1/n}):\Q]},
\end{equation} 
where $P(n)$ denotes the largest prime factor of $n.$
Now we may apply
the Chebotarev density theorem and
we obtain 
that 
\begin{equation}
\label{deltayg}
\cP_g(x)\le 
(\delta_y(g)+\epsilon)\frac{x}{\log x},
\end{equation}
where $\epsilon>0$ is arbitrary and
$x, y$ are sufficiently large (where sufficiently large may depend on the
choice of $\epsilon$).

Completing the sum in 
\eqref{classicalhooley2} and using that
$[\Q(\zeta_n,g^{1/n}):\Q]\gg_{g}n\varphi(n)$ (see \cite[Proposition 4.1]{Wagstaff}) and the classical estimate 
$\varphi(n)^{-1}=O((\log \log n)/n)$,
we obtain that
$$
\delta_y(g)=\delta(g)+O_g\left(\sum_{n\ge y}
\frac{1}{n\varphi(n)}\right)=\delta(g)+O_g\left(\frac{\log \log y}{y} \right).
$$ 
On combining this with 
\eqref{deltayg} we obtain the
estimate
\begin{equation}
\label{deltag}
\cP_g(x)\le 
(\delta(g)+\epsilon)\frac{x}{\log x},
\end{equation}
where $\epsilon>0$ is arbitrary and
$x, y$ are sufficiently large 
 (where sufficiently large may depend on the
choices of $\epsilon$ and $g$). 

For any integer $g \not\in \{-1,0,1\}$ and any integer $t \ge 1$, put
$$
{\mathcal P}(g,t)=\{p: \, 
p\equiv 1 ~({\rm mod~}t),~ \ord_p(g)=(p-1)/t\}.
$$
Now, if the Riemann Hypothesis holds for the number fields
$\Q(\zeta_{nt},g^{1/nt})$ with all square-free $n$,  then 
Hooley's proof can be easily extended, resulting in the
estimate \eqref{hooleytype} 
for the set $\cP(g,t)$  
with density
\begin{equation}
\label{classicalwagstaff} 
\delta(g,t)=\sum_{n=1}^\infty
\frac{\mu(n)}{[\Q(\zeta_{nt},g^{1/nt}):\Q]},
\end{equation} 
and with $\delta(g,t)/A$ a rational number; see \cite{Lenstra}. 
This number was first computed 
explicitly by Wagstaff \cite[Theorem 2.2]{Wagstaff}, 
which can be done much more
compactly and elegantly these days 
using the character sum method of Lenstra et al.~\cite{LMS}.

By Wagstaff's work \cite{Wagstaff} we have
$\delta(\cQ)=\delta(4,2)=3A/2.$ 
Alternatively it is an easy and instructive calculation
to determine $\delta(4,2)$ oneself. Since 
$\sqrt{2}\in \mathbb Q(\zeta_n)$ if and only if $8 \mid n$, we see
that if $4\nmid n,$ then 
$[\Q(\zeta_{2n},2^{1/2n}):\Q]=\varphi(2n)n$ and 
so by \eqref{classicalwagstaff}, 
$$
\delta(4,2)=\sum_{n=1}^{\infty}\frac{\mu(n)}{\varphi(2n)n}=\sum_{2\nmid n}^{\infty}\frac{\mu(n)}{\varphi(n)n}
+\sum_{2\mid n}^{\infty}\frac{\mu(n)}{2\varphi(n)n}=
\frac{3}{4}\sum_{2\nmid n}^{\infty}\frac{\mu(n)}{\varphi(n)n}=\frac{3}{2}A,
$$
where we use the fact that
$$
\sum_{n=1\atop (m,n)=1}^{\infty}\mu(n)f(n)=\prod_{p\nmid m}(1-f(p))  
$$
holds certainly true if the sum is absolutely convergent and $f(n)$ is a multiplicative function defined on the square free integers (cf. 
Moree and Zumalac\'arregui \cite[Appendix A.1]{Zuma}, 
where a similar problem with $g=9$
instead of $g=4$ is considered).

The following result generalizes the above to 
the case where we require the primes in
${\mathcal P}(g,t)$ to also be in some prescribed
arithmetic progression. It follows from Lenstra's work \cite{Lenstra}, who introduced Galois theory
into the subject.

\begin{theorem}
\label{lenstrarekenkundigerij}
Let $1\le a\le d$ be coprime integers. 
Let $t\ge 1$ be an integer. 
Put
$${\mathcal P}(g,t,d,a)=\{p:\, 
p\equiv 1~({\rm mod~}t),~p\equiv a~({\rm mod~}d),~\ord_p(g)=(p-1)/t\}.$$
Let
$\sigma_a$ be the automorphism of $\mathbb Q(\zeta_d)$ determined by
$\sigma_a(\zeta_d)=\zeta_d^a$. Let $c_a(m)$ be $1$ if the restriction
of $\sigma_a$ to the field $\mathbb Q(\zeta_d)\cap \mathbb Q(\zeta_m,g^{1/m})$
is the identity and $c_a(m)=0$ otherwise. Put
\begin{equation*}
\delta(g,t,d,a)=\sum_{n=1}^{\infty}\frac{\mu(n)c_a(nt)}{[\mathbb Q(\zeta_d,\zeta_{nt},g^{1/nt}):\mathbb Q]}.
\end{equation*}
Then, assuming RH for all 
number fields $\Q(\zeta_d,\zeta_{nt},g^{1/nt})$ with
$n$ square-free,
we have
\begin{equation}
\label{hooleygeneral}
{\mathcal P}(g,t,d,a)(x)=\delta(g,t,d,a)\frac{x}{\log x}+
O_{g,t,d}\left(\frac{x\log\log x}{\log^2x}\right),
\end{equation} 

Unconditionally we have the weaker 
statement that
\begin{equation}
\label{deltagtfa}
{\mathcal P}(g,t,d,a)(x)  \le 
(\delta(g,t,d,a)+\epsilon)\frac{x}{\log x},
\end{equation}
where $\epsilon>0$ is arbitrary and
$x$ is sufficiently large (where sufficiently large may depend on the
choice of $\epsilon,g,t,d$ and $a$).
\end{theorem}

It seems that this result has
not been formulated in the literature.
It is a simple combination of two cases each of which
have been intensively studied, namely the primes
having a near-primitive root $(d=1,t>1),$ and
the primes in arithmetic progression having
a prescribed primitive root ($t=1$).

As before $\delta(g,t,d,a)/A$ is a 
rational number that can be explicitly computed.
The case $g=t=2,$ $d=8$ and $a=7$ is one of the most
simple cases. This is a lucky coincidence, as in our proof of 
Proposition \ref{prop:main} 
we will apply Theorem \ref{lenstrarekenkundigerij}  to determine $\delta(\cQ_7)=\delta(2,2,8,7)$.  

Note that 
$
\cP(g,t,d,a)\subseteq \{p:\, p\equiv 1~({\rm mod~}t),~p\equiv a~({\rm mod~}d),~\ord_p(g) \mid (p-1)/t\}.
$
It is shown in \cite[Theorem 1.3]{MS} that if 
the latter set 
is not empty, then it contains a positive density subset of 
primes that are not in ${\mathcal P}(g,t,d,a).$

\section{Proofs of the main results}

It suffices to prove Theorems~\ref{thm:class} and \ref{thm:GB} and Proposition~\ref{prop:main}.

\subsection{Proof of Theorem~\ref{thm:class}}

By \cite[Proposition 3.4]{HK}, we obtain 
\begin{equation*}
h_{p,2}^{-}=(-1)^{\frac{p-1}{2}}2^{2-p}E_{0,\omega_p}E_{0,\omega_p^{3}}\cdots E_{0,\omega_p^{p-2}}.
\end{equation*}
Using Lemma \ref{p4} and \eqref{eq:EnGn}, we then infer that 
\begin{align*}
h_{p,2}^{-} & \equiv (-1)^{\frac{p-1}{2}}2^{2-p}E_{1}(0)E_{3}(0)\cdots E_{p-2}(0) \\
& \equiv \frac{(-1)^{\frac{p-1}{2}}2^{2-p}}{(p-1)!!}G_2 G_4 \cdots G_{p-3} G_{p-1} ~({\rm mod~}p). 
\end{align*}
So, if $p$ is G-irregular, then $p \mid h_{p,2}^{-}$.  
Conversely, if $p \mid h_{p,2}^{-}$ and  $p$ is not a Wieferich prime, 
then by \eqref{eq:BnGn} we first have $p \nmid G_{p-1}$, 
and thus $p$ is G-irregular. 
\qed

\subsection{Proof of Theorem~\ref{thm:GB}}

We first recall a fact about Bernoulli numbers that 
any odd prime $p$ does not divide the denominators of the Bernoulli numbers $B_2, B_4, \ldots, B_{p-3}$ 
(this follows from the von Staudt-Clausen theorem). 
Now, given an odd prime $p$, if it is G-regular, then 
there is no $1\le k \le (p-3)/2$ such that 
$p$ divides the integer $G_{2k}$, that is, $2(1-2^{2k})B_{2k}$ by \eqref{eq:BnGn}; 
and so $p$ is B-regular and $\ord_p(4)=(p-1)/2$. 

Conversely, if $p$ is B-regular and $\ord_p(4)=(p-1)/2$, 
then $p$ does not divide the denominators of the Bernoulli numbers $B_2, B_4, \ldots, B_{p-3}$ 
and $p\nmid 2^{2k}-1$ for 
 $1\le k\le (p-3)/2.$
Consequently $p$ does not divide any integer $G_{2k}=2(1-2^{2k})B_{2k}$ 
with $1\le k\le (p-3)/2$, 
which implies that $p$ is G-regular. \qed

\subsection{Proof of Proposition~\ref{prop:main}}
The proof relies on Theorem \ref{lenstrarekenkundigerij}. 
We only establish the assertion  
under GRH, as the proof of
the unconditional result is very similar. Namely, 
it uses the unconditional estimate \eqref{deltagtfa} instead of
\eqref{hooleygeneral}.

It is enough to prove the result in case $8\mid d.$ 
In fact, in case $8\nmid d$ we lift the congruence class $a~({\rm mod~}d)$
to congruence classes
with modulus lcm$(8,d)$. The ones among those that
are $\not\equiv 1~({\rm mod~}8)$ have 
relative density $AR({\rm lcm}(8,d),a)=AR(d,a)$ (as $R(d,a)$ only depends
on the odd prime factors of $d$). The one that is 
$\equiv 1~({\rm mod~}8)$ (if it exists at all) has relative density zero.
It follows that the relative density of the unlifted congruence 
equals $c(d,a)R(d,a)A$ with $c(d,a)$ the relative density  of the primes
$p\not\equiv 1~({\rm mod~}8)$ in the congruence class 
 $a~({\rm mod~}d).$ The easy determination of $c(d,a)$ is left
 to the interested reader.

From now on we assume that $8\mid d.$ 
We can write $a\equiv j~({\rm mod~}8)$ for some
$j\in \{1,3,5,7\}$ and 
distinguish three cases. 

{\it Case I: $j=1$.}  By Proposition \ref{prop:splitsing} the set $\cQ(d,a)$ is empty and the result
holds trivially true.

{\it Case II: $j\in \{3,5\}$.} By Proposition \ref{prop:splitsing},
$$\cQ(d,a)=\{p:~p\equiv a~({\rm mod~}d),~\ord_p(2)=p-1\}.$$
By Theorem \ref{lenstrarekenkundigerij}, under GRH, this
set has density $\delta(2,1,d,a).$ For arbitrary
$g,d,a$ the third author determined the rational number
$\delta(g,1,d,a)/A,$ see \cite[Theorem 1]{Mor1} or 
\cite[Theorem 1.2]{Mor2}. On applying his result, the proof of this subcase is
then completed.

{\it Case III: $j=7$.} By Proposition \ref{prop:splitsing},
$$\cQ(d,a)=\{p:~p\equiv a~({\rm mod~}d),~\ord_p(2)=(p-1)/2\}.$$
For simplicity we write $\delta=\delta(\cQ(d,a)).$
By Theorem \ref{lenstrarekenkundigerij} we have
\begin{equation}
    \label{startertje}
\delta=\delta(2,2,d,a)=
\sum_{n=1}^{\infty}
\frac{\mu(n)c_a(2n)}{[\mathbb Q(\zeta_d,\zeta_{2n},2^{1/2n})
:\mathbb Q]}.
\end{equation}
In case $n$ is even, then trivially
$\mathbb Q(\sqrt{-1})\subseteq \mathbb Q(\zeta_d)\cap 
\mathbb Q(\zeta_{2n},2^{1/2n}).$ 
As $\sigma_a$ acts by conjugation
on $\mathbb Q(\sqrt{-1}),$ cf. \cite[Lemma 2.2]{Mor2}, and
not as the identity, it follows that $c_a(2n)=0.$

Next assume that $n$ is odd and square-free.
Then by \cite[Lemma 2.4]{Mor2} we infer that
$$\mathbb Q(\zeta_d)\cap \mathbb Q(\zeta_{2n},2^{1/2n})
=\mathbb Q(\zeta_{(d,n)},\sqrt{2}).$$
Since 
$$
\sigma_a \big |_{\mathbb Q(\sqrt{2})}=\text{id.} \qquad \text{and} \qquad 
\sigma_a \big |_{\mathbb Q(\zeta_{(d,2n)})}
\begin{cases}
= \text{id.} & \text{if~}a\equiv 1~({\rm mod~}(d,2n));\\
\ne \text{id.} & \text{otherwise},
\end{cases}
$$
we conclude that
$$
c_a(2n) = 
\begin{cases}
 1 & \text{if~}a\equiv 1~({\rm mod~}(d,2n));\\
0 & \text{otherwise},
\end{cases}
$$
with `id.'\,a shorthand for identity.
Note that the assumptions on $a,d$ and $n$ imply that 
$a\equiv 1~({\rm mod~}(d,2n))$ iff
$a\equiv 1~({\rm mod~}2(d,n))$ iff
$a\equiv 1~({\rm mod~}(d,n)).$
It follows that \eqref{startertje} simplifies to
$$
\delta=
\sum_{\substack{2\nmid n\\ a\equiv 1~({\rm mod~}(d,n))}}
\frac{\mu(n)}{[\mathbb Q(\zeta_d,\zeta_{2n},2^{1/2n})
:\mathbb Q]}.
$$
When $n$ is odd and square-free, using \cite[Lemma 2.3]{Mor2} we obtain  
$$
[\mathbb Q(\zeta_d,\zeta_{2n},2^{1/2n})
:\mathbb Q]=[\mathbb Q(\zeta_{\text{lcm}(d,2n)},2^{1/2n})
:\mathbb Q]=
n\varphi(\text{lcm}(d,2n))=n\varphi(\text{lcm}(d,n)).$$
We thus get
$$
\varphi(d)\delta=\sum_{\substack{2\nmid n\\ a\equiv 1~({\rm mod~}(d,n))}}
\frac{\mu(n)\varphi(d)}{n\varphi(\text{lcm}(d,n))}.
$$
Put 
$$w(n)=\frac{n\varphi(\text{lcm}(d,n))}{\varphi(d)}.$$
In this notation
we obtain
$$
\delta=\frac{1}{\varphi(d)}
\sum_{\substack{2\nmid n\\ a\equiv 1~({\rm mod~}(d,n))}}
\frac{\mu(n)}{w(n)},
$$
where the argument in the sum is multiplicative
in $n.$
Using \cite[Lemma 3.1]{Mor2} and the notation used there and
in \cite[Theorem 1.2]{Mor2}, 
we find

\begin{equation*}
\begin{split}
\varphi(d)\delta & = S(1)-S_2(1)=2S(1)=2A(a,d,1)\\
&=2A\prod_{p|(a-1,d)}(1-\frac{1}{p})\prod_{p|d}
\left(1+\frac{1}{p^2-p-1}\right)=\delta(d,a),
\end{split}
\end{equation*}
as was to be proved. \qed

\section{Outlook}
A small improvement of 
the upper bound 
\eqref{eq:underGRH2} 
(and consequently the lower
bound \eqref{dainequality}) would be possible if 
instead of the estimate \eqref{deltagtfa}
a Vinogradov
type estimate for $\cP(g,t,d,a)(x)$ could be established, say
\begin{equation} 
\label{vinogradovtypeadvanced}
  \cP(g,t,d,a)(x)  \le \delta(g,t,d,a)\frac{x}{\log x} + O_{g,t,d}\left(\frac{x(\log\log x)^2}{\log^{5/4} x} \right). 
\end{equation} 
Vinogradov \cite{Vino} established the above result in
case $a=d=t=1$.    
Establishing \eqref{vinogradovtypeadvanced} seems technically quite involved. Recent work
by Pierce et al.~\cite{effectChebo} 
offers perhaps some hope that one can even improve
on the error term in \eqref{vinogradovtypeadvanced}.

\section{Some numerical experiments}  \label{sec:data}

In this section, using the Bernoulli numbers modulo $p$ function developed by David Harvey in  Sage~\cite{Sage} 
(see \cite{BH,HHO,Harvey10} for more details and improvements) and the euler{\_}number function in Sage, 
 we provide numerical evidence for the truth of Conjectures~\ref{conj:Siegel2}, \ref{conj:E-irre2}, \ref{con:global} and \ref{con:local} 
and also for \eqref{eq:underGRH} in Proposition~\ref{prop:main}. 

The Bernoulli numbers modulo $p$ function returns the values of $B_0, B_2, \ldots, B_{p-3}$ modulo $p$, 
and so by checking whether there is a zero value we can determine whether $p$ is B-irregular. 
For checking the E-irregularity, we use the euler{\_}number function in Sage to compute and store Euler numbers 
and then use the definition of E-irregularity. 
It would be a separate project to test large E-irregular primes, 
cf.\,\cite{BH,HHO,Harvey10}. 

In the tables, we only record the first six digits of the decimal parts. 

In Tables~\ref{tab:B-irre} and 
~\ref{tab:E-irre} the ratio $\cP_B(d,a)(x)/\pi(x;d,a),$  
respectively $\cP_E(d,a)(x)/\pi(x;d,a)$ is recorded for 
$x=10^5$ in the column `experimental' for
various choices of $d$ and $a$, and in the column `theoretical'  
the limit value  
predicted by Conjecture~\ref{conj:Siegel2} is given.

\begin{table}   
\centering
\caption{The ratio $\cP_B(d,a)(x)/\pi(x;d,a)$ for $x=10^5$}
\label{tab:B-irre}
\begin{tabular}{|c|c|c|c|}
\hline
$d$ & $a$    & experimental & theoretical  \\ \hline \hline  
     
3 & 2 & 0.394424  &  \\ \cline{1-2}

4 & 1 & 0.388877  &  \\ \cline{1-2}

5 & 4 & 0.397071  &   \\ \cline{1-2}

7 & 4 &  0.391005 & 0.393469  \\ \cline{1-2}

9 & 8 & 0.387742  &  \\ \cline{1-2}

12 & 5 & 0.390203  &  \\ \cline{1-2}

15 & 13 & 0.389858  &  \\ \cline{1-2}

20 & 13 &  0.385191 &  \\ \hline

\end{tabular}
\end{table}

\begin{table}   
\centering
\caption{The ratio $\cP_E(d,a)(x)/\pi(x;d,a)$ for $x=10^5$}
\label{tab:E-irre}
\begin{tabular}{|c|c|c|c|}
\hline
$d$ & $a$    & experimental & theoretical  \\ \hline \hline  
     
3 & 2 &  0.395672 &  \\ \cline{1-2}

4 & 1 & 0.388040  &  \\ \cline{1-2}

5 & 4 & 0.397071  &   \\ \cline{1-2}

7 & 4 &  0.393504 & 0.393469  \\ \cline{1-2}

9 & 8 & 0.391494  &  \\ \cline{1-2}

12 & 5 & 0.388127  &  \\ \cline{1-2}

15 & 13 & 0.399002  &  \\ \cline{1-2}

20 & 13 & 0.385191 &  \\ \hline

\end{tabular}
\end{table}

Table~\ref{tab:G-irre1} gives the ratio 
$\cP_G(x)/\pi(x)$ for various values of $x,$ and the value in the column `theoretical' is the limit value  
$1-3A/(2\sqrt{e})$ predicted by Conjecture~\ref{con:global}.

\begin{table}   
\centering
\caption{The ratio $\cP_G(x)/\pi(x)$}
\label{tab:G-irre1}
\begin{tabular}{|c|c|c|}  
\hline
$x$ & experimental & theoretical  \\ \hline \hline

$10^5$ & 0.661592  & \\ \cline{1-2}

$10^6$ & 0.659558  & \\ \cline{1-2}

$2\cdot 10^6$  & 0.660860  & 0.659776 \\ \cline{1-2}

$3\cdot 10^6$  & 0.661413  & \\ \cline{1-2}

$4\cdot 10^6$  & 0.660683  & \\ \cline{1-2}

$5\cdot 10^6$ &  0.660864  & \\ \hline

\end{tabular}
\end{table}

Table~\ref{tab:G-irre2} gives the ratio $\cP_G(d,a)(x)/\pi(x;d,a)$ for 
$x=5 \cdot 10^6$ in the column `experimental' for
various choices of $d$ and $a$, and the corresponding
limit values $1-c(d,a)R(d,a)A/\sqrt{e}$ 
predicted by Conjecture~\ref{con:local} 
are in the column `theoretical'.

\begin{table}   
\centering
\caption{The ratio $\cP_G(d,a)(x)/\pi(x;d,a)$ for $x=5 \cdot 10^6$}
\label{tab:G-irre2}
\begin{tabular}{|c|c|c|c|}
\hline
$d$ & $a$ & experimental & theoretical  \\ \hline \hline

3 & 1 & 0.728296  & 0.727821 \\ \hline

5 & 2 & 0.643010  & 0.641870 \\ \hline

4 & 1 & 0.771512  & 0.773184 \\ \hline

20 & 9 &  0.757311 & 0.761246 \\ \hline

12 & 11 & 0.460584  & 0.455642 \\ \hline

20 & 19 & 0.528567  & 0.522493 \\ \hline

8 & 7 & 0.550086  & 0.546368 \\ \hline

24 & 13 & 0.634191  & 0.637094 \\ \hline
 
\end{tabular}
\end{table}

Finally, Table~\ref{tab:distrioverAP} gives the ratio
${\mathcal Q}(d,a)(x) / \pi(x;d,a)$
for  $x= 5 \cdot 10^6$ in the column `experimental' for
various choices of $d$ and $a$.  
In the column `theoretical', there is
the corresponding relative 
density $\delta(d,a)$ 
predicted in \eqref{eq:underGRH} and known to be true under GRH. 

\begin{table}   
\centering
\caption{The ratio $\cQ(d,a)(x)/\pi(x;d,a)$ for $x=5 \cdot 10^6$}
\label{tab:distrioverAP}
\begin{tabular}{|c|c|c|c|}
\hline
$d$ & $a$ & experimental & theoretical  \\ \hline \hline

3 & 1 & 0.449049  & 0.448746 \\ \hline

5 & 2 & 0.589614  &  0.590456  \\ \hline

4 & 1 & 0.374664  & 0.373955 \\ \hline

20 & 9 & 0.395498  & 0.393637 \\ \hline

12 & 11 & 0.898284  & 0.897493 \\ \hline

20 & 19 & 0.789316  & 0.787275 \\ \hline

8 & 7 & 0.747300  & 0.747911 \\ \hline

24 & 13 & 0.598815  & 0.598329 \\ \hline
\end{tabular}
\end{table}

In view of the definition of
the constant $c(d,a)$ in Proposition~\ref{prop:main}, 
there are four cases excluding the case $8 | d$ and $a \equiv 1 ~({\rm mod}~8)$ (which gives $c(d,a)=0$). 
For each of these four cases there are two instances in Tables~\ref{tab:G-irre2} and \ref{tab:distrioverAP}.

\section*{Acknowledgement}
The authors would like to thank the referee for careful reading and valuable comments.  
This work was supported by the National Natural Science Foundation  of China, Grant 
No.\,11501212. 
The research of Min-Soo Kim and Min 
Sha was also supported by the Kyungnam University Foundation Grant, 2017, 
respectively a Macquarie University Research Fellowship. 
The authors thank Bernd Kellner for pointing out a link with the Genocchi numbers
and suggesting the references \cite{Ern2, HP, Holden}, 
and Peter Stevenhagen for very helpful feedback. 
They also thank Alexandru Ciolan for proofreading earlier versions.

\end{document}